\documentclass{amsart}

\usepackage{epstopdf}
\usepackage[latin1]{inputenc}
\usepackage[T1]{fontenc}
\usepackage{amsmath,amssymb,amsthm}
\usepackage{graphicx}
\usepackage{epsfig}
\usepackage{hyperref}

\begin{document}

\newtheorem{theorem}{Theorem}[section]
\newtheorem{lemma}[theorem]{Lemma}
\newtheorem{corollary}[theorem]{Corollary}
\newtheorem{cor}[theorem]{Corollary}
\newtheorem{proposition}[theorem]{Proposition}
\theoremstyle{definition}
\newtheorem{definition}[theorem]{Definition}
\newtheorem{example}[theorem]{Example}
\newtheorem{claim}[theorem]{Claim}
\newtheorem{xca}[theorem]{Exercise}

\theoremstyle{remark}
\newtheorem{remark}[theorem]{Remark}

\newcommand{\pint}{P_{\text{int}}}
\newcommand{\contractible}{contractible}
\newcommand{\girth}{g}
\newcommand{\id}{\mbox{Id}_{[3]}}
\newcommand{\B}{\partial(D)}
\newcommand{\Z}{\mathbb Z}
\newcommand{\N}{\mathbb N}
\newcommand{\F}{\mathbb F}
\newcommand{\Q}{\mathbb Q}
\newcommand{\R}{\mathbb R}
\newcommand{\Ztwo}{\mathbb Z_2}
\newcommand{\T}{T}
\newcommand{\1}{{\bf 1}}
\newcommand{\BB}{\partial(\T)}
\newcommand{\Btilde}{\partial(\tilde D)}
\newcommand{\BBtilde}{\partial(\tilde \BB)}
\newcommand{\Bprime}{\partial(D')}
\newcommand{\len}{L}
\newcommand{\diam}{\text{diam}}
\newcommand{\be}{\begin{equation}}
\newcommand{\ee}{\end{equation}}
\newcommand{\prob}{\mbox{\bf P}}
\newcommand{\E}{\mbox{\bf E}}
\newcommand{\vint}{V_{\text{int}}}
\newcommand{\dint}{D_{\text{int}}}
\newcommand{\fdis}{F_{\text{dis}}}
\newcommand{\vdis}{V_{\text{dis}}}
\newcommand{\fsid}{F_{\text{dis}}}
\newcommand{\vsid}{V_{\text{dis}}}
\newcommand{\basin}{{\text{Basin}}}
\newcommand{\Star}{{\text{Star}}}
\newcommand{\atc}{associated triangulated disk}
\newcommand{\ptc}{triangulated disk}

\newcommand\HFILL{\hspace*{\fill}}
\newcommand\VFILL{\vspace*{\fill}}

\newenvironment{pf} {\noindent{\sc Proof. }}{{\hfill
$\Box$}\par\vskip2\parsep}

\newenvironment{pfofthm}[1]
{\par\vskip2\parsep\noindent{\sc Proof of Theorem\ #1. }}{{\hfill
$\Box$}
\par\vskip2\parsep}

\newenvironment{pfoflem}[1]
{\par\vskip2\parsep\noindent{\sc Proof of Lemma\ #1. }}{{\hfill
$\Box$}
\par\vskip2\parsep}

\newenvironment{skpf} {\noindent{\sc Sketch
Proof. }}{{\hfill $\Box$}\par\vskip2\parsep}
\newenvironment{skpfof}[1] {\par\vskip2\parsep\noindent{\sc Sketch
Proof of\ #1. }}{{\hfill $\Box$}\par\vskip2\parsep}

\newcommand{\blankbox}[2]{%
  \parbox{\columnwidth}{\centering
    \setlength{\fboxsep}{0pt}%
    \fbox{\raisebox{0pt}[#2]{\hspace{#1}}}%
  }%
}

\title{The fundamental group of random $2$-complexes}

\author[E. Babson]{Eric Babson}
\author[C. Hoffman]{Christopher Hoffman}
\author[M. Kahle]{Matthew Kahle}
\address{Department of Mathematics, University of Washington, Seattle, WA 98195}

\date{\today.}

\thanks{This research supported in part by NSA grant \#
H98230-05-1-0053, NSF grant \# DMS-0501102, and NSF-VIGRE grant \# DMS-0354131. }

\maketitle


\begin{abstract}
We study Linial-Meshulam random $2$-complexes $Y(n,p)$, which are 
simplicial complexes with $n$ vertices and $n \choose 2$ edges, where each possible two dimensional face
is included independently with probability $p$.
We find the threshold for simple connectivity
of these complexes to be roughly $p = n^{-1/2}$.  This is in sharp contrast to the
threshold for vanishing of the first homology with finite coefficient group, which was shown
by Linial and Meshulam to be $p = 2 \log n / n$.

We use a variant of Gromov's local-to-global theorem
for linear isoperimetric inequalities
to show that when $p = O( n^{-1/2 -\epsilon}$) the fundamental group is
word hyperbolic.
Along the way we classify the homotopy type of sufficiently sparse $2$-dimensional
simplicial complexes and establish isoperimetric inequalities for such complexes. These intermediate
results do not involve randomness.  

\end{abstract}

\section{Introduction}
\label{Sect1}

In this article we find the threshold for simple connectivity
of the random $2$-dimensional simplicial complexes $Y(n,p)$ introduced
by Linial and Meshulam \cite{Nati} to be roughly $p = n^{-1/2}$.  One motivation
for this is continuing the thread of probabilistic topology initiated
by Linial and Meshulam \cite{Nati}, and even earlier by Erd\H{o}s and R\'enyi \cite{Erd1}. (Other recent work
concerning the topology of random simplicial complexes can be found in \cite{Kahle_neighborhood,Kahle_clique,MW,Pipp}.)

Another motivation for this study is the connection to the random groups studied in geometric group theory \cite{Ollivier}.
In fact we use geometric group theory techniques to show that in the sparse regime
the fundamental group is hyperbolic on the way to showing that it is nontrivial;
in particular we apply Gromov's local-to-global principle for linear
isoperimetric inequalities.

Erd\H{o}s and R\'enyi initiated the now vast subject of random graphs with their edge-independent model $G(n,p)$ \cite{Erd1}.

\begin{definition} The {\it Erd\H{o}s-R\'enyi random graph} $G(n,p)$ is the probability space of all graphs on
vertex set $[n]:=\{1,2,\ldots,n\}$ with each of the ${ n \choose 2}$ possible edges included independently with probability $p$. We say $G(n,p)$ {\bf asymptotically almost surely (a.a.s.)} has property $\mathcal{P}$ if
$\lim_{n\rightarrow \infty}\prob(G(n,p) \in \mathcal{P}) =1$.
\end{definition}

A seminal result 
is that $p=\log{n}/n$ is a sharp threshold for the connectivity
of the random graph. 

\begin{theorem}[Erd\H{o}s and R\'enyi  \cite{Erd1}] 
\label{theorem ER} Let $\omega(n) \rightarrow \infty$ as $n
\rightarrow \infty$.
\begin{enumerate}
\item If $p=(\log{n}-\omega(n))/n$ then $G(n,p)$ is a.a.s. disconnected, and 
\item if $p=(\log{n}+\omega(n))/n$ then $G(n,p)$ is a.a.s. connected.
\end{enumerate}
\end{theorem}

Nathan Linial and Roy Meshulam exhibited a $2$-dimensional
homological analogue of Theorem \ref{theorem ER}. They defined
 a model of random 2-dimensional simplicial complexes $Y(n,p)$ to be the
probability space of simplicial complexes on vertex set $[n]$ and
edge set ${[n]}\choose{2}$, with each $2$-face appearing
independently with probability $p$.

\begin{theorem}[Linial-Meshulam \cite{Nati}]
\label{theorem LM} Let $\omega(n) \rightarrow \infty$ as $n
\rightarrow \infty$. If $p=(2\log{n}-\omega(n))/n$ then a.a.s.\ $H_{1}(Y,\Z / 2 \Z) \not= 0$, and if $p =
(2\log{n}+\omega(n))/n$ then a.a.s.\ $H_{1}(Y,\Z / 2 \Z )= 0$.
\end{theorem}
Meshulam and Wallach extended this result to
$H_{d-1}(Y,\Z / q\Z)$ for arbitrary primes $q$ and $d$-dimensional complexes \cite{MW}.

Our first result is that when $p$ is sufficiently large, $\pi_1(Y(n,p))$ a.a.s.\ vanishes.

\begin{theorem}\label{theorem 1}
Let $\omega(n) \rightarrow \infty$ as $n \rightarrow \infty$.
If $$p\geq \left( \frac{3\log{n}+\omega(n)}{n} \right) ^{1/2}$$
then a.a.s.\ $\pi_1(Y(n,p))=0$.
\end{theorem}

Our main result and most of the work of this paper is to show that the exponent $1/2$ in Theorem
\ref{theorem 1} is best possible.

\begin{theorem}\label{theorem 2}
For any constant $\epsilon > \frac12$, if $$p =O \left( n^{-\epsilon} \right)$$
then  $\pi_1(Y(n,p))$  is a.a.s.\ hyperbolic and nontrivial.
\end{theorem}

The proof of Theorem \ref{theorem 2} relies on general notions of negative curvature due to Gromov.  As the Linial-Meshulam result is an analogue of the Erd\H{o}s-R\'enyi theorem, our result is analogous to certain thresholds for random groups.  The random group seemingly closest to what we study here is the following triangular model.

\begin{definition}
Let $0 \le d \le 1$. A {\it triangular random group}
on $n$ generators at density $d$ is the group presented by $H=<b_1, \ldots, b_n \mid R>$
where $R = \{ r_1, r_2, \dots, r_t \}$, and each $r_i$ is chosen i.i.d.\ uniformly from the $T=2n(2n-1)^2$  reduced words of length $3$, and $t= \lfloor T^d \rfloor$. 
\end{definition}

\.Zuk characterized the threshold for vanishing of $H$ as $n \to \infty$.

\begin{theorem}\label{theorem 2 triangular}\cite{Zuk} If $d<1/2$ then $H$ is a.a.s.\ nontrivial hyperbolic,
and if $d>1/2$ then $H$ is a.a.s.\ trivial.
\end{theorem}

\begin{theorem}\label{maintopology}
If $X$ is a finite connected two dimensional simplicial
complex such that
$$ 2f_0(W) > f_2(W)$$ for all nonempty subcomplexes $W \subset X$
then $X$ has the homotopy
type of a wedge of circles, spheres and real projective planes.
Thus the fundamental group of $X$ is a free product of $\Z$'s and $\Z/ 2\Z$'s.
\end{theorem}

We use this to obtain a linear isoperimetric inequality for null-homotopic loops in $X$. (We precisely define cycles $\gamma$ and the notions of length $L(\gamma)$ and area $A(\gamma)$  on page \pageref{sasha} in Section \ref{prf_thm2}.)

\begin{theorem}\label{mainiso}
For any $\epsilon>0$ there is $\beta>0$ such that if $X$ is a finite, two dimensional simplicial
complex with
$$ (2-\epsilon)f_0(W)>f_2(W)$$ for all nonempty subcomplexes $W \subset X$
then every contractible cycle $\gamma$ satisfies
$$L(\gamma)>\beta A(\gamma).$$
\end{theorem}

We also need a version of Gromov's general principle that one can go from local linear isoperimetric inequalities to global ones \cite{gromov1}, a method which has been very useful in the study of random groups.

The rest of the paper is organized as follows. Section \ref{prf_thm1} contains the proof of Theorem \ref{theorem 1}.   Section \ref{prf_thm2} contains the outline of the proof of Theorem \ref{theorem 2}.  In Section \ref{homotopy} we prove Theorem \ref{maintopology}.  We use this in Section \ref{oilspill} to prove Theorems \ref{theorem 2} and \ref{mainiso}.  
The appendices prove a technical lemma and the version of Gromov's local to global principle that we need.

\section{Proof of Theorem \ref{theorem 1} } \label{prf_thm1}

If $X$ is a two dimensional
simplicial complex and $v \in F_0(X) $ is a vertex define the {\bf
link} of $v$, denoted $\mbox{lk}_X(v)$, to be the one dimensional
simplicial complex (graph) with
$$F_0(\mbox{lk}_X(v))=\{\{p\}|\{ v,p \}\in F_1(X)\}.$$ and
$$F_1(\mbox{lk}_X(v))=\{\{p,q\}|\{ v,p,q \}\in F_2(X)\}.$$

The key observation necessary to prove Theorem \ref{theorem 1} is the following.
\begin{lemma}
\label{special}
For any $a,b,c \in [n]$ and simplicial complex $Y$ such that
\begin{enumerate}
\item $\mbox{lk}_{Y}(a) \cap \mbox{lk}_{Y}(b)$ is connected and
\item there is $d \in [n]$ such that $\{a,b,d\} \in F_2(Y)$
\end{enumerate}
then the 3-cycle $\{\{a,b\},\{a,c\},\{b,c\}\}$ bounds an embedded disk in $Y$.
\end{lemma}
\begin{proof}

Since $\mbox{lk}(a) \cap \mbox{lk}(b)$ is connected, there is a sequence $\{x_i\}_{1}^{k}$ such that $c = x_1$, $d = x_k$ and $\{x_i,x_{i+1}\} \in \mbox{lk}(a) \cap \mbox{lk}(b)$ for all $i<k$. The edge $\{x_i, x_{i+1}\} \in \mbox{lk}(a)\cap \mbox{lk}(b)$ iff $\{\{a, x_i, x_{i+1}\}, \{b, x_i, x_{i+1}\}\}\subseteq F_2(Y)$. So we see that $\{\{a,b\},\{a,c\},\{b,c\}\}$ bounds an embedded
disk, as in Figure \ref{fig:simply}.

\begin{figure}
\begin{centering}
\includegraphics[width=2.5in]{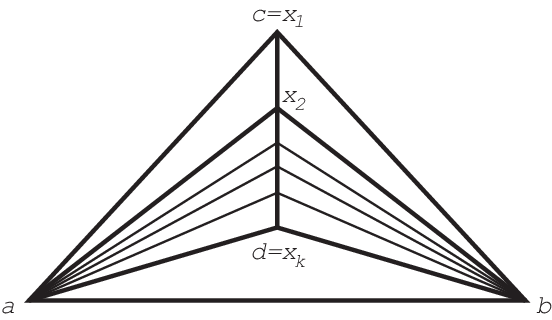}
\end{centering}
\caption{Because $\mbox{lk}(a) \cap \mbox{lk}(b)$ is connected and edge $ab$ is contained in at least one face, the $3$-cycle $abc$ bounds a disk.}
\label{fig:simply}
\end{figure}

\end{proof}


Note that for each pair of vertices $a,b \in [n]$ the distribution of
$$\mbox{lk}_{Y}(a) \cap \mbox{lk}_{Y}(b)$$ is identical to the
Erd\H{o}s-R\'enyi random graph $G(n-2,p^2)$.
To complete the proof of Theorem \ref{theorem 1} we show that if $p$ is sufficiently large then the hypothesis of Lemma \ref{special} are a.a.s.\ satisfied for every distinct $a, b \in [n]$.  This requires bounding the probability that $G(n,p)$ is not connected when $p$ is a bit larger than the threshold of $\log(n)/n$.  The calculation is slightly messy so we delegate the proof to Appendix 1.  

\begin{lemma} \label{TY}
Let $\omega(n) \rightarrow \infty$ as $n \rightarrow \infty$.
If $p= \left( \frac{3\log{n}+\omega(n)}{n} \right) ^{1/2}$ then a.a.s.
\begin{enumerate}
\item $ \mbox{lk}_{Y}(a) \cap \mbox{lk}_{Y}(b) \text{ is connected }  $ and
\item there exists $d \in [n]$ such that $\{a,b,d\} \in F_2(Y)$
\end{enumerate}
for all distinct $\{a, b\} \subseteq [n]$  a.a.s.
\end{lemma}

\begin{pfofthm}{\ref{theorem 1}}
By Lemmas \ref{special} and \ref{TY} we have that a.a.s.\ every 3-cycle is
\contractible. That $Y$ is a.a.s.\ simply connected follows as $F_1(Y)$ is the complete graph and  every $k$-cycle in the fundamental group is a product of 3-cycles.
\end{pfofthm}

A more complicated version of this argument was used in \cite{Kahle_clique} to prove vanishing of $k$th homology $H_k$ for arbitrary $k$, for a different kind of random simplicial complex.

\section{Outline of Theorem \ref{theorem 2}} \label{prf_thm2}

\subsection{Notation}
We will work with simplicial maps between simplicial complexes.  
For a two dimensional simplicial complex $X$ we
write $F_0=F_0(X)$, $F_1=F_1(X)$ and $F_2=F_2(X)$ for the sets of
vertices, edges and faces of $X$ and $f_i=|F_i|$ for the
respective numbers. For an edge $e \in F_1(X)$ we write
$f^e_2=f^2_e(X)=|\{t \in  F_2(X): e \subset
\partial(t)\}|$ for the number of 2-faces containing $e$ in their
boundaries.

\label{sasha}
\begin{definition}
We define $C_r$ to be the length $r$ cycle with $F_0(C_r)=[r]=\{1, \ldots , r\}$
($[0]=\emptyset$) and $$F_1(C_r)=\bigcup_{i=1}^{r-1} \bigg\{ \{i,i+1\}   \bigg\} \cup\bigg\{\{r,1\}\bigg\}.$$
\end{definition}

\begin{definition}
Similarly we define $I_r$ to be the length $r$ path with $F_0(I_r)=\{0,1, \ldots , r\}$
and $$F_1(I_r)=\bigcup_{i=0}^{r-1} \bigg\{ \{i,i+1\}   \bigg\}.$$
\end{definition}

\begin{definition}
Let $\gamma:C_r\to X$. We say $(C_r\xrightarrow{b}D\xrightarrow{\pi}X)$ is a  {\bf filling}
of $\gamma$ if $\gamma=\pi b$ and the mapping cylinder $\text{Cyl}(b)$ of $b$
is a disk. 
This condition on $b$ is equivalent to a simplicial Van Kampen diagram.  
\end{definition}

\begin{definition} \label{length}
Define the {\bf length} of a path $\gamma:I_r\rightarrow X$ or a loop 
$\gamma:C_r\rightarrow X$ to be $L(\gamma)=r$ (or $\epsilon r$ in the 
scaled situation of section \ref{appendix2}).
\end{definition}

\begin{definition} \label{area}
 Define the {\bf area} of a curve $\gamma$ to be
$$A(\gamma)=\hbox{min}\{f_2(D)|  \hbox{ $(C\xrightarrow{b}D\xrightarrow{\pi}X)$ is a filling of }\gamma\}$$
if $\gamma$ is contractible (or $\epsilon^2f_2(D)$ in the scaled situation of section \ref{appendix2}) 
and $A(\gamma)=\infty$ if $\gamma$ is not contractible. We say that a filling $(C\xrightarrow{b}D\xrightarrow{\pi}X)$ of
$\gamma$ is {\bf minimal} if $A(\gamma)=f_2(D)$.
\end{definition}

\begin{figure}
\begin{centering}
\includegraphics[width=4.5in]{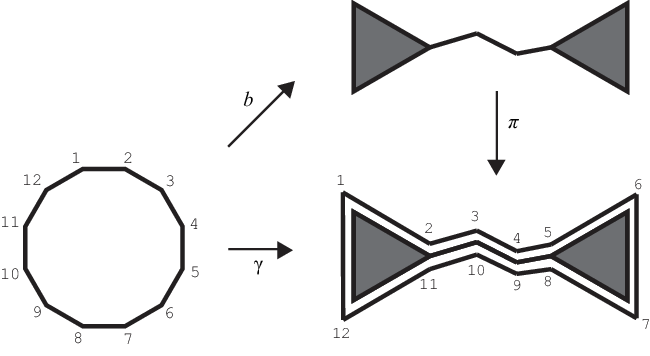}
\end{centering}
\caption{A minimal filling of a $12$-cycle with $A ( \gamma) =2$.}
\label{fill}
\end{figure}

\subsection{Sketch of proof of Theorem \ref{theorem 2}}
Write $\id:[3]\to [3]$ for the identity map.
We show that for a typical $Y$ (with probability approaching
1) the cycle $\id$ is not contractible and thus $Y$ is not simply
connected. The main step is to prove a linear isoperimetric
inequality. This means that there is
$\rho$ such that for a typical $Y$ any
contractible loop $\gamma:C\to Y$ satisfies
\begin{equation} \label{newlipe}
A(\gamma) \leq \rho L(\gamma).
\end{equation}
Once we have a linear isoperimetric inequality for a typical $Y$
we have
$$\prob(3\rho<A(\id)<\infty)\to 0.$$
Then we complete the proof by showing that
$$\prob(A(\id) \leq 3\rho)\to 0.$$
To carry out this program we introduce the
following definitions. The parameter $e$ determines the density 
at which a finite complex will appear in a random complex.  
Throughout this section, $X$ is a $2$-complex with vertex set $F_0(X)=[n]$.

\begin{definition}
We write $$e(X)=\min_{Z\subseteq X}\left({f_0(Z) \over
  f_2(Z)}\right).$$
More generally if  $[w]\subseteq
F_0(X)$ then write
 \begin{equation}\nonumber
e_w(X)=\min_{ \substack{Z\subseteq X \\ [w] \subset
F_0(Z)}}\left({f_0(Z)-w \over
  f_2(Z)}\right).
  \end{equation}
We say  $X$  is {\bf $\epsilon$-admissible} if $e(X)\geq \epsilon$. For some $w \leq n$
  we say $X$ is {\bf $(\epsilon,w)$-admissible} if $e_w(X)\geq \epsilon$.
  We say a $2$-complex $X$ is {\bf admissible} ({\bf $w$-admissible}) 
if there is some $\epsilon>\frac12$ such that $X$ is $\epsilon$-admissible
($(\epsilon,w)$-admissible).
We define things generally for convenience in notation, but in fact we will only ever use the case $w=3$.
\end{definition}

The following lemma is the first step in showing the existence of a linear isoperimetric
inequality.
\begin{lemma} \label{huji}
For every $\epsilon>\frac12$ there is $\lambda$ such that for every $X$
with $e(X)\geq \epsilon$, every contractible loop $\gamma:C\rightarrow X$ satisfies 
\begin{equation}\label{lipe}
A(\gamma) \leq \lambda L(\gamma).\end{equation}
\end{lemma}

The proof of Lemma \ref{huji} appears on page \pageref{malia} and requires
the use of several other lemmas in between.
The key to proving Lemma \ref{huji} is to analyze the topology of
admissible complexes. In Lemma \ref{topology} we show
that every admissible two complex is homotopy
equivalent to a wedge product of circles, spheres and projective
planes.

We cannot apply Lemma \ref{huji} directly to get a linear
isoperimetric inequality for $Y$ because for a typical $Y$ we have
that $f_2(Y) \geq O( n^2)$ (since we may assume $\epsilon < 1$) and
$f_0(Y)=n$.  Thus $e(Y)=O(\frac1n)$.  Instead we analyze the subcomplexes  $X
\subset Y$ with $f_2(X)$ small. The next lemma tells us which small
subcomplexes can be embedded in a typical $Y$.

\begin{definition} For simplicial complexes $Z$ and $X$ with
$F_0(Z) \cup F_0(X) \subset \Z^+$, and with
$[w] \subseteq F_0(Z) \cap F_0(X)$, a {\bf $w$-inclusion} $g$ of $Z$ into $X$ is an
injective simplicial map $g:Z \to X$ such that $g(i)=i$ for all $1\leq i \leq w$.
\end{definition}

\begin{definition}
Let $X$ be a simplicial complex with $F_0(X)=[n]$ for some $n$.
$X$ is {\bf $(\epsilon, m)$-sparse}
if for every $2$-complex $Z$ with
\begin{enumerate}
\item $f_2(Z)\leq m$ and \item $f_0(Z) < \epsilon f_2(Z)$
\end{enumerate}
there is no embedding of  $Z$ in $X$.
$X$ is called {\bf $(\epsilon, m,3)$-sparse}
if  for every $2$-complex $Z$ with
\begin{enumerate}
\item[(3)]  $[3]\subseteq
F_0(Z)$,
\item[(4)] $f_2(Z)\leq m$ and \item[(5)]  $f_0(Z) -3< \epsilon f_2(Z)$
\end{enumerate}
there is no $3$-inclusion
of $Z$ into $X$.
\end{definition}

\begin{lemma} \label{derandomize}
For every $m \in \Z^+$, $\epsilon > \frac12 $, and $p=O(n^{-\epsilon})$, we have that $Y \in Y(n, p)$ is  $(\epsilon, m,
3)$-sparse a.a.s.
\end{lemma}
\begin{proof}
For fixed $m$ and $r$ there are only finitely many complexes $Z$ with $f_2(Z)<m$. Thus to prove that
$Y$ is $(\epsilon,m,3)$-sparse a.a.s.\  we only
need to prove that for any given complex $Z$ which does not satisfy conditions
3, 4 and 5 that
\begin{equation}\label{pcmi}
 \prob\bigg(\text{$Z$ has a $3$-inclusion in $Y$ }\bigg)=0\ \ a.a.s.
\end{equation}
If $Z$ satisfies conditions 3, 4 and 5 then choose $\alpha>0$ with 
$$f_0(Z)-3-\epsilon f_2(Z)<-\alpha$$
and hence
\begin{eqnarray*}
\prob\big(\text{$Z$ has a $3$-inclusion into $Y$}\big)
     &\leq & \E \big(\text{number of $3$-inclusions of $Z$ into $Y$}\big)\\
     &\leq &n^{f_0(Z)-3}p^{f_2(Z)}\\
     &\leq &n^{f_0(Z)-3}C^{f_2(Z)}n^{-\epsilon f_2(Z)}\\
     &< &C^m n^{-\alpha}.
\end{eqnarray*}
\end{proof}

We establish a linear isoperimetric inequality for $Y$ by combining Lemmas \ref{huji} and \ref{derandomize}, together with Gromov's local to global principle. Similar results for groups appear in \cite{gromov2} and \cite{Pap} but we require the result for $2$-dimensional simplicial complexes, so we include a proof in Appendix 2 for the sake of completeness.

\begin{theorem}  \label{gromov}
If $\rho\geq 4$ and $X$ is a finite simplicial complex for which every loop
$\gamma:C_r\rightarrow X$ with $A(\gamma)  \leq  44\rho^2$  satisfies $A(\gamma) \leq {\rho\over 44} r$ 
then every
contractible loop $\gamma:C_r\rightarrow X$ satisfies $A(\gamma) \leq \rho r$.
\end{theorem}

The local to global principle gives us the following.

\begin{lemma} \label{huji3}
For every $\epsilon>\frac12$ there are $m$ and $\rho$ such that every
contractible loop $\gamma:C_r \to X$ in an $(\epsilon,m)$-sparse complex $X$ 
satisfies $$A(\gamma)<\rho L(\gamma)$$
and if $X$ is also $(\epsilon,m,3)$-sparse then the loop 
$\hbox{Id}_{[3]}:C_3\rightarrow X$ is not contractible.  
\end{lemma}

\begin{proof}
For the first part, given $\epsilon>\frac12$ choose $\lambda$ as in Lemma \ref{huji} and then
use $m=(44)^3(\lambda)^2$ and $\rho=\hbox{max}\{4,44\lambda\}$ in Theorem \ref{gromov}.  

For the second part assume there is a minimal filling 
$C_3\xrightarrow{b}D\xrightarrow{\pi}X$ of $\hbox{Id}_{[3]}$
and with $\lambda$ as above take $m=\hbox{max}\{44^3\lambda^2,3(44)\lambda\}$.
By the first part, 
$f_2(\hbox{Im}(b))\leq f_2(D)=A(\hbox{Id}_{[3]})\leq 3\rho \leq m$ so 
$e_3(\hbox{Im}(b))>\frac12$.  Lemma \ref{hiji2} below now gives a contradiction.  
\end{proof}

The same technology that we use to prove Lemma \ref{huji} can also
be used to prove Lemma \ref{hiji2} on page \pageref{proof_hiji2}.

\begin{lemma} \label{hiji2}
For every $X$ such that $[3]\subseteq F_0(X)$ with $e_3(X) > {1\over 2}$ the curve $\id$ is not contractible in $X$.
\end{lemma}

\begin{pfofthm}{\ref{theorem 2}}
That $\pi_1(Y)$ is nontrivial follows from Lemmas
\ref{derandomize} and \ref{huji3}.

That it is hyperbolic
follows from Lemmas \ref{derandomize} and \ref{huji3}, as follows.  If there is a linear isoperimetric inequality on $Y$,
then there is a linear isoperimetric inequality on $\pi_1 (Y)$ as well; since $Y$ is compact, $\pi_1(Y)$ is quasi-isometric to $\widetilde{Y}$, the universal cover of $Y$.  Groups which satisfy a linear isoperimetric inequality also satisfy a ``thin triangles'' condition and are Gromov hyperbolic \cite{gromov1}.
\end{pfofthm}

\section{Homotopy type of admissible 2-complexes} \label{homotopy}


The following lemma is a strengthening of Theorem \ref{maintopology}.

\begin{lemma}\label{topology}
If $X$ is an admissible, finite, connected and two dimensional simplicial
complex then $X$ has the homotopy
type of a wedge of circles, spheres and real projective planes.
(Thus $\pi_1 ( X ) $ is isomorphic to a free product of $\Z$s and $\Z/ 2\Z$s and is hyperbolic).

Moreover there is a subcomplex $Z\subseteq X$ with
$F_1(Z)=F_1(X)$ for which the inclusion induces an isomorphism of
fundamental groups
and $\chi (Z')\leq 1$ for any
connected subcomplex $Z'\subseteq Z$.
\end{lemma}

The proof of Lemma \ref{topology} requires several other intermediate results,
and appears on page \pageref{proof_topology}.

\subsection{Stratified complexes and webs}

The proof of Lemma \ref{topology} is by induction.
We assume by way of contradiction that there is a
minimal counterexample and make reduction moves to find a smaller one.
However, there is a fairly serious complication in that
the reduction moves do not always leave us with a simplicial complex. For this reason we introduce
the following more general complexes.

For a compact manifold with boundary $M$ we use the notation
$\partial M$ for the boundary of $M$,  and $M^\circ$ for the interior.
(For a $0$-dimensional $M_0$,
$\partial M_0 = \emptyset$ and $M_0^\circ = M_0$).


\begin{definition}
A ($2$-dimensional) {\bf stratified complex} $N$ consists of
 \begin{enumerate}
\item a topological space $N$ homeomorphic to the realization of a
finite simplicial $2$-complex,
 \item for each $i \in \{0, 1, 2 \}$, a compact $i$-dimensional manifold with boundary $M_i$ (not necessarily connected), and
 \item immersions  $\psi_i :M_i \rightarrow N$ such that the restrictions to interiors $\psi_i |_ {M_i^\circ}$ are embeddings,
the images $\psi_i(  M_i^\circ)$ partition $N$, and  $\psi_i(\partial
M_i)\subseteq \psi_{i-1}(M_{i-1})$.
\end{enumerate}

We call the connected components of $M_i$ the $i$-dimensional {\bf faces} of $N$ and use
the upper index to distinguish them. The set
of $i$-dimensional faces of $N$ is denoted by $F_i(N)$ so that $M_i = \cup_{\phi\in F_i(N)}M_i^\phi$.
\end{definition}

We refer the reader to Figure \ref{fig:strata} for an example of a stratified complex $N$.  Note that the structure is not quite the same as a CW-complex, since the cells need not be topological disks and the $\psi_i$ are required to be immersions.

\begin{definition} \label{shout}
If $N$ is a stratified complex, $i < i'$, $u\in F_i(N)$ and
$u'\in F_{i'}(N)$  then write:
\begin{enumerate}
\item  $f_i(N)=|F_i(N)|$,\\
\item  $f_{u'}^{u}(N)=f_{u}^{u'}(N)=|\psi_i^{-1}(x) \cap M^{u'}_{i'}|$ for any $x \in \psi_i( M^{\circ u}_i)$,\\
\item  $f_{u'}^{i}(N)=\sum_{u \in F_i(N)}f_{u'}^{u}(N)$ and\\
\item  $f_{i'}^{u}(N)=\sum_{u' \in F_{i'}(N)}f_{u'}^{u}(N)$.\\
 \end{enumerate}
\end{definition}

\begin{figure}
\begin{centering}
\includegraphics[width=3.75in]{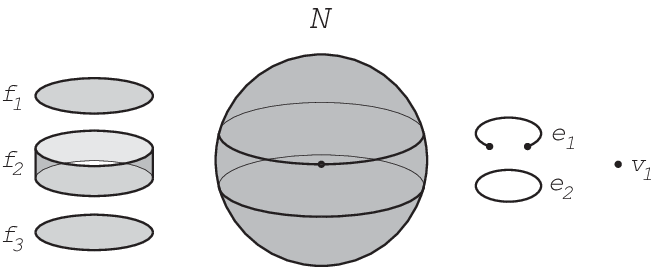}
\end{centering}
\caption{A stratified complex $N$, homeomorphic to a sphere. In this example $f_0(N) = 1$, $f_1(N)=2$, and $f_2(N)=3$, and also have, for example, that $f^v_1 = 2$. $f^{e_1}_2 = 2$, $f^v_2 =2.$ }
\label{fig:strata}
\end{figure}

For every stratified complex $N$ and $e \in F_1(N)$, $M_1^e$ is homeomorphic to either 
an interval or a circle. At times we
need to distinguish these cases, and so introduce the following additional notation.

\begin{definition} We write
\begin{enumerate}
\item $F_{1,c}(N)=\{e\in F_1(N)|M_1^e\cong S^1\}\subseteq F_1(N)$,
\item $f_{1,c}(N)=|F_{1,c}(N)|$ and
\item for a face $u'\in F_2(N)$ write $f_{u'}^{1,c}=\sum_{u \in F_{1,c}(N)}f_{u'}^{u}(N)$ 
\end{enumerate}
\end{definition}

\begin{definition} \label{webdefinition}
A {\bf web} $W$ is a stratified complex with an absolutely continuous measure $\mu$ on
$M_1$ which induces via $\psi$ a measure (also $\mu$) on
$\partial M_2$ (since $\psi|_{\partial M_2}$ is a covering map) 
with $\mu (M_1^e) \in \N$ for every $e\in F_1(W)$.
This is equivalent to a positive integral length function on the $1$-faces.

A {\bf subweb} $W'$ of a web $W$ is uniquely specified by subsets
$F_i(W') \subseteq F_i(W)$. If $W$ is a web and $v\in F_0(W)$ is a
vertex, then $\hbox{lk}_W(v)$ is the link of $v$ in $W$ which is
a stratified one dimensional complex, with
$f_i(\hbox{lk}_W(v))=f^v_{i+1}(W)$.
\end{definition}

\begin{definition}
If $X$ is a finite simplicial complex then $W=W(X)$ is the {\bf
associated} web with $|X|=|W|$, $F_i(W)=F_i(X)$, $M_i(W)=F_i(X)
\times \Delta_i$ (a disjoint union of standard simplices), $\psi$
restricted to each face is an embedding and every edge $e\in F_1$
has length one ($\mu M_1^e = 1$). Thus we can consider simplicial
complexes as special cases of webs.
\end{definition}

Two very useful functions from webs to integers are the Euler
characteristic $\chi (W)=\chi (|W|)$ and the length
 $$L(W)= 2\mu (M_1(W)) - \mu(\partial M_2(W)).$$
\begin{definition}
We say that a nonempty web $W$ is {\bf admissible} if every nonempty
subweb $W'$ satisfies $$(2\chi + L)(W') > 0.$$
\end{definition}
\begin{definition}
Other useful functions from webs to integers
include
\begin{enumerate}
\item $d(W)=\hbox{max}\{i|f_i(W)>0\}$ (dimension),
\item $\delta (W)=\hbox{min}\{f_i^u|u \in F_{i-1}(W), 0<i\leq d(W)\}$ (minimum degree) and
\item $\girth (W)=\hbox{min}\{\mu (S)|f:S
\rightarrow \psi M_1 \hbox{ is an isometric embedding of a
circle}\}$ (girth).
\end{enumerate}
\end{definition}

\begin{definition}
A web $W$ is a {\bf refinement} of another web $W'$ (which is a 
{\bf coarsening} of $W$) if there is a
homeomorphism $r:|W|\rightarrow |W'|$ and for each $u\in F_i(W')$
there is a subweb $W^u$ of $W$ with the restriction of $r$ to
$|W^u|\subseteq |W|$ a homeomorphism onto $\psi(M_i^u)$.  
The measure $\mu'$ is induced by $\mu$.  
\end{definition}

Up to isomorphism, a refinement depends
only on the subwebs $\{\{ W^u\}_{u\in F_i(W')}\}_i$.

\begin{lemma} \label{tahini}
If $X$ is a simplicial $2$-complex, $W=W(X)$ is the associated web 
and $W'$ is a proper coarsening then
\begin{enumerate}
\item $\girth (W) \geq 3$,
\item if $X$ is admissible then so is $W$,
\item $\chi(M_2(W'))=\sum_{t \in F_2(W')} \chi(M_2^t)$,
\item $\chi(M_0(W'))=f_0(W')$,
\item[(5)] $|W'|\cong|W|$,
\item[(6)] $d(W')=d(W)$,
\item[(7)] $\girth (W')=\girth (W)$,
\item[(8)]$\delta (W)=\hbox{min}\{\delta (W'), 2\}$ and
\item[(9)] if $W$ is admissible
then so is $W'$.
\end{enumerate}
\end{lemma}
\begin{proof}
These all follow directly from the definitions.
\end{proof}

\begin{lemma} \label{gosto}
If $W$ is a $2$-dimensional stratified complex so that no
vertex link decomposes as a (nontrivial) wedge sum with a circle as one of the summands,
then there is a unique coarsening $C(W)$ of $W$ such that
$2\not\in\{f_i^u\}_{i\in\{1,2\},u\in F_{i-1}(W)}$ (so $\delta (C(W)) \not= 2$).
\end{lemma}

\begin{proof}
The construction of $C(W)$ follows.  Uniqueness is clear.

Set $M_{0,s}(W)=\{v\in M_0(W)||\{e\in F_1^v|f^2_e\not=2\}|=s\}$ 
(vertices in $s$ singular edges).

Set $M_0(C(W))=\bigcup_{s\not\in \{0,2\}}M_{0,s}(W)$ (vertices which are neither smooth nor part of a 1-dim stratum).

Set
$$M_1(C(W))=(\bigcup_{e\in F_1(W), f_2^e\not=2}M_1^e\bigcup M_{0,2}(W))\slash \sim$$
where $a\sim v$ if $a\in M_1(W)$, $v\in
M_{0,2}(W)$ and $\psi(a)=\psi(v)$ (singular edges glued at their ends if the singularities along the edges agree).

Set $$M_2(C(W))=M_2(W)\cup_{e\in F_1(W), f_2^e=2}M_1^e\cup
M_{0,0}(W)\slash \sim$$ where $a\sim b$ if $\psi(a)=\psi(b)$ and
either $a\in M_1(W)$ and $b\in M_0(W)$ or $a\in M_2^t$ and $b\in
M_1^e$ and there is an inclusion of $M_1^e$ into $M_2^t$ commuting
with $\psi$ (all 2-faces glued together at smooth points).

The map $\psi(C(W)):M(C(W))\rightarrow |C(W)|=|W|$ is then inherited
from $\psi$ for $W$.

It is now straightforward to check that each $M_i(C(W))$ is a
manifold with boundary, with interior points precisely the
equivalence classes of points in the interior of some $M_j(W)$.

The other properties are straightforward to check.
\end{proof}

\begin{definition}
A graph is {\bf $2$-connected} if it has at least three vertices and
is connected after deleting any vertex.
\end{definition}

\begin{lemma} \label{newlemma}
If $W$ is a web with all vertex links $2$-connected then
the hypothesis of Lemma \ref{gosto} holds.
\end{lemma}

\begin{proof}
The hypothesis of Lemma \ref{gosto} is that no vertex link of $W$ decomposes as a (nontrivial) wedge sum with a circle as one of the summands. If some vertex link $\mbox{lk}(v)$  did have such a decomposition, then by definition of wedge sum there would be a cut vertex in $\mbox{lk}(v)$, contradicting the assumption that $\mbox{lk}(v)$ is $2$-connected.
\end{proof}



Now we introduce a collapsing construction.
\begin{definition} \label{cambria} If $A \subseteq W$ is a subcomplex of a two
dimensional stratified complex $W$ then write $K=K_A(W)$ for the
maximal subcomplex of $W$ for which every edge $e\in F_1(K)$
either has $f_2^e(K)\geq 2$ or  $f_2^e(K)\geq 1$ and $e\in
F_1(A)$.
\end{definition}
This collapsing construction is useful in this section with
$A=\emptyset$ and  again in the next section with more general
$A$.

\begin{lemma}\label{transformers}
For any web $W$ and $A \subset W$ each connected component of $|W|$
has the homotopy type of a wedge of components of $|K_A(W)|$ and
circles. Also $\delta(K_\emptyset(W))\geq 2.$
\end{lemma}
\begin{proof} $K_A(W)$ is obtained from $W$ by a sequence
of collapses of cells to wedges of circles (removing an edge in
exactly one face) which induce homotopy equivalences and deletions
of edges contained in no faces.  The second statement follows
directly from the definition.
\end{proof}

The main geometric fact about admissible webs is the following.  
The proof of Lemma \ref{eric1} appears on page \pageref{eric1pf}. 
\begin{lemma} \label{eric1}
If $W$ is an admissible, connected and $2$-dimensional web with $\girth (W) \geq
3$
then $|W|$ has the
homotopy type of a wedge of circles, spheres and projective
planes.
\end{lemma}

\begin{definition}
The following partial order on webs will be used for induction.  
For $(i,j), (i',j') \in \Z^2$ we say that
$(i,j)$ is before $(i',j')$ in the order if either $i>i'$ or $i=i'$ and $j<j'$
(so the order looks like $(2,0)\dots(1,5)(1,4)\dots(1,0)\dots(0,5)(0,4)\dots(0,0)$). For any web $W$
let
$$f_{i,j}(W)=\big|\{u\in F_i(W)|f_{i+1}^u=j\}\big|$$
(so $f_2(W)=f_{2,0}(W)$ and $f_1(W)=\sum_jf_{i,j}(W)$).
For two webs $W$ and $W'$ write $W<W'$ if
$(i,j)$ is the first pair (in the above order) such that
$f_{i,j}(W) \neq f_{i,j}(W')$ and $f_{i,j}(W) < f_{i,j}(W').$
\end{definition}

The proof of Lemma \ref{eric1} requires a few intermediate lemmas, and is by induction with respect to the partial order we just defined.  Whenever we refer to a minimal counterexample it is minimal with respect to this partial order.


\begin{lemma} \label{eric2}
If $W$ is a minimal counterexample to Lemma \ref{eric1}
then $\delta(W) \geq 3$ and all vertex links are $2$-connected.  
\end{lemma}

\begin{proof}
We begin by showing that $W=K_\emptyset(W)$. 
Note that $K_\emptyset(W)$ is a subweb of $W$ 
and hence admissible,
no larger than $W$ (with first difference occurring in $f_2=f_{2,0}$ 
or else $f_{1,0}$) and has $W$ homotopy equivalent to the wedge sum
of components of $K_\emptyset(W)$ and some circles. Thus
$W=K_\emptyset(W)$ by minimality and by Lemma \ref{transformers}
$\delta(W)\geq 2$.

Finally we show that vertex links in $W$ are $2$-connected and hence that 
$W=C(W)$.  
If $W$ has a vertex link which is not
$2$-connected then splitting a vertex into two along a cut point
of its link associated to an edge $e\in F_1(W)$ gives a complex $W'$
with smaller $f_{1,f^e_2}$ (without changing $f_2$ or any earlier 
$f_{1,j}$) which is still
admissible and homotopy equivalent to $W$. Admissibility of $W'$ follows by 
considering the quotient map $q:W'\rightarrow W$ and noting that 
if $W'$ were inadmissible and $K'\subseteq W'$ had $(2\chi+L)(K')\leq 0$ 
then $(2\chi+L)(q(K'))\leq 0$ so that $W$ would also be inadmissible.

Splitting a vertex $v$ into two between
connected components of its link gives a complex $W'$ with smaller $f_{0,f^v_1}$
(again without changing any earlier $f$) which is similarly still admissible and 
$W$ is homotopy equivalent to the
wedge of the connected components of $W'$ if $W'$ is disconnected and to the 
wedge of $W'$ and a circle otherwise.  

Thus by minimality all vertex links of $W$ are $2$-connected and by Lemma
\ref{newlemma} the hypotheses of Lemma \ref{gosto} are satisfied and $C(W)$ exists.  
Note that $C(W)$ is no larger than and homeomorphic to $W$ so that by minimality 
$W=C(W)$.  

Thus we have $2\leq\delta(K_\emptyset(W)=\delta(W)=\delta(C(W))\neq 2$. 
\end{proof}

\begin{definition}
A $2$-face $t\in F_2(W)$ with $M_2^t$ a disk is a {\bf digon} if $f_t^1=2$ 
and a {\bf monogon} if $f_t^1=1$ and $f_t^{1,c}=0$.  
\end{definition}

\begin{lemma} \label{eric3}
If $W$ is a minimal counterexample to Lemma \ref{eric1}
then $W$ has no digons or monogons.
\end{lemma}

\begin{proof}
Note that if $t\in F_2(W)$ then $f^1_t>0$.  Otherwise $M_2^t=|W|$ is a 
connected $2$-manifold with (by admissibility) positive Euler characteristic
Thus $W$ is a sphere or projective plane and not a counterexample.

If $t$ is a digon with $F_t^1=\{e,f\}$ and $\mu(e)\leq\mu(f)$
then construct $W'$ by choosing a
homeomorphism $\tau:M_1^e\rightarrow M_1^f$ compatible with $\psi$
on the boundaries and setting $M_0(W')=M_0(W)$,
$M_1(W')=M_1(W)\slash \{a\sim\tau(a)\}$ and
$M_2(W')=(M_2(W)-M_2^t)\slash\sim$ where any two points of any $\tau$-folded interval
in $\partial M_2(W)$ are equivalent under $\sim$. 
An interval $I\subseteq \partial M_2(W)$ is {\bf $\tau$-folded} if for every $p\in I$ 
there is $q\in I$ and $a\in M_1^e$ with $\{\psi(p),\psi(q)\}=\{\psi(a),\psi(\tau a)\}$.  
(This operation can shorten or eliminate the boundary of several of the $M_2^t$).  
If $t$ is a monogon with $F_t^1=\{f\}$ treat it as a digon with $\mu(e)=0$ 
so that $W'=W\slash t$ is obtained by contracting $F_2^t$ to a point.  

Note that the homotopy
type of $|W|$ is the wedge sum of the components of
$|W'|$ and some circles. Take $W''$ to be a connected component of
$W'$ which is not a wedge of circles, spheres and projective
planes. Clearly $f_2(W'')<f_2(W)$. 
The measure $\mu''$ is inherited from $\mu$ except on $(e\cup f)\slash\sim$, 
where it agrees with $\mu(e)$. $W''$ can now be seen to be admissible 
by noting that if not there is $K''\subseteq W''$ with $(2\chi +L)(K'')\leq 0$
and hence $(2\chi +L)(K)\leq 0$ if $K\subseteq W$ is the closure of the 2-faces 
of $W$ associated to those of $W''$ in $K''$ along with the face $t$ 
if its boundary ($\psi(\partial t)$) is in $K$.  

Thus by minimality, $W$ has no digons or monogons.
\end{proof}

\begin{definition}
Define $\overline{\mu}$ to be a measure with
$\overline{\mu}(M_1^e)=1$ if $e\in F_1(W)\setminus F_{1,c}(W)$ and
$\overline{\mu}(M_1^e)=3$ if $e\in F_{1,c}(W)$.
\end{definition}

\begin{lemma}
\label{brewedawakenings}
If $W$ is a minimal counterexample to Lemma \ref{eric1} then either
\begin{enumerate}
\item there is $u \in F_0(W)$ with
 $$\sum_{t\in F_2(W)}f^u_t\left({2\chi (M_2^t)\over
  \overline{\mu}(\partial M_2^t)}-1\right)>-2$$ or
 \item there is $u\in F_{1,c}(W)$ with
 $$\sum_{t\in F_2(W)}f^u_t\left({2\chi (M_2^t) \over
  \overline{\mu}(\partial M_2^t)}-1\right)>-2.$$
\end{enumerate}
\end{lemma}

\begin{proof}
For every $t\in F_2(W)$,
\be  \label{tether} 1=\frac{f^t_0+3f^t_{1,c}}
    {\overline{\mu}(\partial M_2^t)}\ee
    and
    \be \chi(M)=\chi (M_0) - \chi (M_1) + \chi (M_2). \label{bark}
\ee
 Let
 \be \label{lbar}
 \overline L(N)= 2\overline \mu (M_1(N)) - \overline \mu(\partial M_2(N)).
 \ee
 Since $\girth (W)\geq
3$ we have $\mu\geq \overline{\mu}$.  Since $\delta (W)\geq 3$ and
$$L=\sum_{e\in F_1(W)}(2-f^2_e)\mu (M_1^e),$$ we have
\be L(W)\leq \overline{L}(W). \label{mark} \ee

 Finally we note that the definition of $\bar \mu$ gives us that
 \be \label{neil}
 (2\bar \mu-2\chi)(M_1(W))=6f_{1,c}(W)
 \ee
 and
 \be \label{koblitz}
 \bar \mu(\partial M_2(W))=\sum_{t \in F_2(W)} f_0^t +3f_{1,c}^t
 \ee

Note that
\begin{eqnarray}
0
 &<& (2\chi + L)(W)\nonumber\\
 & &\hspace{2.75in} \text{definition of admissibility} \nonumber \\
 &\leq& (2\chi + \overline{L})(W)\nonumber \\
 &&\hspace{4in} \text{(\ref{mark})}\nonumber\\
 &=& 2\chi (M_0(W)) + (2\overline{\mu}-2\chi)(M_1(W))
        +(- \overline{\mu}(\partial M_2(W)) + 2\chi(M_2(W))) \nonumber\\
 &&\hspace{3.5in} \text{(\ref{bark}) and (\ref{lbar}}) \nonumber \\
 &=& 2f_0(W) + 6f_{1,c}(W) - \sum_{t\in F_2(W)}(f_0^t+3f^t_{1,c}) + \sum_{t\in F_2(W)}\frac{f^t_0  + 3f^t_{1,c}}{\overline{\mu}(\partial M_2^t)}
        (2\chi)( M_2^t) \nonumber\\
        && \hspace{.6in}\text{term by term equalities using and (\ref{neil}), (\ref{koblitz}) and Lemma \ref{tahini} }\nonumber\\
 &=& \sum_{u\in F_0(W)}\left[2+\sum_{t\in F_2(W)}f^t_u\left({2\chi (M_2^t) \over
  \overline{\mu}(\partial M_2^t)}-1\right)\right]\label{perfil1}\\
 && + \sum_{u\in F_{1,c}(W)}3\left[2+\sum_{t\in F_2(W)}f^t_u\left({2\chi (M_2^t) \over
  \overline{\mu}(\partial M_2^t)}-1\right)\right]. \label{perfil2}\\
        && \hspace{3in} \text{from Definition \ref{shout} }\nonumber
  \end{eqnarray}

Since the sum is positive at least one of the summands in (\ref{perfil1}) or (\ref{perfil2})
must  be positive.   Rearranging the summands completes the proof.
\end{proof}

\begin{lemma} \label{makessense}
If there is $u \in F_{1,c}(W)$ with
 $$\sum_{t\in F_2(W)}f^t_u\left({2\chi (M_2^t)\over
  \overline{\mu}(\partial M_2^t)}-1\right)>-2$$
  then $W$ is not a minimal counterexample to Lemma \ref{eric1}.
\end{lemma}

\begin{proof}
Fix such a $u\in F_{1,c}(W)$ .
For a face $t\in F_2(W)$ to contribute more than $-f_u^t$ to the
sum in (\ref{perfil2}) we must have $\chi
(M_2^t)>0$ which implies that $M_2^t$ is a disk.

By Lemma \ref{eric2} there are at least $f_u^2 \geq \delta
(W)\geq 3$ (weighted) terms, including either two embedded disks
$t,t'\in F^2_u$ with $\partial M_2^t=\partial M_2^{t'} = \psi
(M_1^u)$ or the entire complex is the union of a projective plane
with a disk along an embedded circle.  In the former case deleting
$t$ gives $W'$ which is clearly a smaller admissible
counterexample, contradicting minimality. In the latter case the
entire complex has the homotopy type of a sphere and thus it is
not a counterexample.
\end{proof}

\begin{lemma} \label{check}
If there is $u\in F_0(W)$ with
 $$\sum_{t\in F_2(W)}f^t_u\left({2\chi (M_2^t) \over
  \overline{\mu}(\partial M_2^t)}-1\right)>-2.$$
  then $W$ is not a minimal counterexample to Lemma \ref{eric1}.
\end{lemma}

\begin{proof}
Fix such a $u\in F_{0}(W)$. 
For a face $t\in F_2(W)$ to contribute more than $-f_u^t$ to the
sum in (\ref{perfil1}) we must have $\chi
(M_2^t)>0$ which implies that $M_2^t$ is a disk.

Then there are at least $f_u^2 \geq \frac92$
(weighted) terms, including two embedded disks $t,t' \in F^2_u$
with $\mu (\partial t) = \mu (\partial t') = 3$ and $\mu (\partial
t \cap \partial t')\in \{2,3\}$. One sees this by explicitly
enumerating all ways to get a positive term with at least $3$
vertices and $5$ edges in the link of $u$. It turns out that the
link must be a triangle with two edges doubled and at least $4$ of
the edges must come from triangles and hence all $5$ must be
embedded. Let $t$ and $t'$ be two triangles forming a double edge
in the link of $u$. If $\partial t = \partial t'$, deleting $t$
gives a smaller counterexample contradicting minimality as above.

If $\partial t \not= \partial t'$ then a web $W'$ with the same
homotopy type as $W$ and one fewer two face exists.  $W'$ is obtained by deleting $t$
and identifying the two edges in $(\partial t \cup
\partial t') \setminus(\partial t \cap \partial t')$.
It remains to check that $W'$ is admissible, contradicting the
minimality of $W$ and completing the proof. Checking admissibility
is straightforward.
\end{proof}

\begin{pfoflem}{\ref{eric1}}\label{eric1pf}
Assume that $W$ is a minimal counterexample.  Lemmas \ref{brewedawakenings}, \ref{makessense} and \ref{check} form a contradiction.
\end{pfoflem}
\vspace{.2in}

\begin{pfoflem}{\ref{topology}} \label{proof_topology}
The first claim follows from Lemma \ref{eric1} applied to $W(X)$. 

To prove the second claim choose $j:Z\rightarrow X$
to be a minimal subcomplex such that $\pi_1(j)$ is an isomorphism
and $f_1(Z)=f_1(X)$. If $Z\simeq S^2\vee Z'$ then choose
a simplicial map $f:S\rightarrow Z$ with $|S|\cong S^2$ and
$H_2(f;\Z/2\Z)\not=0$ and a 2-face $t$ of $Z$ with $|f^{-1}(t)|$
odd. Fix a presentation of
$$\pi_1(Z\setminus t)=<a_1, \ldots a_s, b_1, \ldots b_t|b_i^2=1\ \forall i =1,\dots,t>$$ and
express some element of $\pi_1(Z\setminus t)$ represented by the
boundary of $t$ as a cyclically reduced word $[\partial
t]=w_1w_2\ldots w_v$.   The restriction $f|_{S\setminus
f^{-1}(t)}$ shows that $[\partial t]^r=1$ for some odd $r$ so
$w_1=w_v^{-1}$ and hence $v\leq 1$ and $[\partial t]=1$. Note that
$$\pi_1(i):\pi_1(Z\setminus
t)\rightarrow \pi_1(Z\setminus t)/<[\partial t]>=\pi_1(Z)$$ is the
quotient map (where $<\ldots>$ is the normal closure) and an
isomorphism, contradicting the minimality of $Z$.

\end{pfoflem}

\section{Isoperimetric inequalities} \label{oilspill}
 Classifying the homotopy type of admissible complexes
$X$ is a major step towards establishing a linear isoperimetric
inequality for $Y$. However we also need a bound on the number  of
faces in the spheres and projective planes. (A family of spheres
with an increasing number of vertices need not satisfy any one
linear isoperimetric inequality.)

To get this bound we now recall the function $L$ (previously
defined for webs) which generalizes the length of the boundary of
a disk.
$$L(X)=2f_1(X)-3f_2(X)=\sum_{e \in F_1(X)}(2-f^2_e).$$

\begin{lemma}\label{popped}
If $X$ is a $w$-admissible $2$-complex then
$$f_2(X) \leq {2\chi (X)-2w+L(X) \over 2e_w(X)-1}.$$
\end{lemma}

\begin{proof}
By the definitions of $\chi$, $e_w$ and $L$ we get
$$\left[\begin{array}{rrr}1&-1&1\\-1&0& e_w(X)\\
0&2&-3\end{array}\right]\left[\begin{array}{c}f_0(X)\\
    f_1(X)\\f_2(X)\end{array}\right]
    \leq\left[\begin{array}{c}\chi (X)\\-w\\L(X)\end{array}\right].$$
Multiplying both sides on the left by $[2,2,1]$ gives the desired
result.
\end{proof}

\begin{lemma}\label{below}
For every $\epsilon>{1\over 2}$ there is $\beta>1$ with the following property.
If $X$ is an $\epsilon$-admissible connected $2$-complex with $L(X)\leq 0$
and $\chi (X)\leq 1$ then any contractible loop $\gamma:C_{r}\to X$
satisfies $$A(\gamma)<\beta L(\gamma).$$
\end{lemma}

\begin{proof}
As $L(X)\leq 0$ and $\chi(X) \leq 1$,  Lemma \ref{popped} implies
that $f_2(X)\leq {2\over 2\epsilon-1}$. Since $L(X) \leq 0$ we also have $f_1(X)={1\over 2}L(X)+{3\over 2}f_2(X)\leq{3\over 2\epsilon-1}$, and since $X$ is connected $f_0(X)\leq f_1(X)+1$.  So for any fixed $\epsilon$ there are only a finite number of
$X$ that satisfy the hypothesis of the lemma. By Lemma \ref{topology} we have that $\pi_1(X)$ is a free product of $\Z$ and $\Z /2\Z$ terms.  
A free product of hyperbolic groups is hyperbolic, so $\pi_1(X)$ is hyperbolic. (Here we mean ``word hyperbolic'' in the sense of Gromov \cite{gromov1}). Hence we have a linear isoperimetric inequality on $\pi_1(X)$ with respect to any finite presentation of the group.

Let $T$ denote any spanning tree of $X$, and let $X/T$ denote the quotient of $X$ with all the points of $T$ identified to a single point. We can endow $X/T$ with the structure of a
CW-complex with one vertex, and $X/T$ is easily seen to be a {\it presentation complex} for $\pi_1(X)$.

Let $\pi: X \to X/T$ be the quotient map. By assumption, $\gamma (C_r)$ is null homotopic in $X$ so $b(C_r)=\pi \gamma (C_r)$ is a trivial word in $\pi_1(X)$.  Hence there is some constant $\beta_X$ independent of $\gamma$ such that 
$$A(\gamma) < \beta_X  L(b (C_r)).$$
We also have that $$L (b (C_r)) \le L(\gamma)$$ 
since some edges may get contracted in the quotient, but the cycle can not get longer.

Thus there is a $\beta_X$ such that
$$A(\gamma)<\beta_X r.$$ for all null homotopic curves $\gamma$ in $X$.
As there are only finitely many such $X$ we can set
$\beta = \max_{X} \beta_X$, and we have that for all $\gamma$ and $X$
$$A(\gamma)<\beta r,$$
as desired.
\end{proof}

\subsection{Proofs of Lemmas \ref{huji}
and \ref{hiji2}}

\begin{definition}
A simplicial map $f:X \rightarrow Y$ is a {\bf $d$-immersion}
 if the restriction to the closed star
of any $d$-face is an embedding.
\end{definition}

\begin{lemma}\label{embedding}
If $(C_r\xrightarrow{b}D\xrightarrow{\pi}X)$ is a minimal filling and $\gamma=\pi b$ is a $0$-immersion
then $\pi$ is a $1$-immersion.
\end{lemma}
\begin{proof}
Assume not and choose vertices $u\not=v$ of $D$ with $\pi(u)=\pi(v)$ and
$B(1,u)\cap B(1,v)$ containing at least one edge.
Take $x$ and $y$ to be the two outer points of $B(1,u)\cap B(1,v)$ 
(that is the two points in either $\hbox{Im}(b)$ or else the closure of a component 
of $D-B(1,\{u\})-B(1,\{v\})$ intersecting $\hbox{Im}(b)$).  
Consider the length 4 loop $\delta:C_4\rightarrow D$ with vertices $v,x,u$ and $y$ 
and interior open disk $D''$ a component of $D-\hbox{Im}(\delta)$.  
Construct a smaller filling
$$(C_r \xrightarrow{b} D' \xrightarrow{\pi|_{D'}},X)$$
of $\gamma$ where
$$D'=(D- D'')\slash (u\sim v)$$
is the quotient simplicial complex.  
\end{proof}

Let $(C_r\xrightarrow{b}D\xrightarrow{\pi}X)$ be a minimal
filling.  To show for \ref{huji} that $f_2(D)\leq\rho r$ we will
define a subcomplex $D_{L\leq 0} \subset D$.
Then we break the proof up into two cases.  Lemma \ref{berkeleyone} will cover the case that
$f_2(D_{L \leq 0})/f_2(D)$ is close to 1. Lemma \ref{berkeleytwo} will cover the case that
$f_2(D_{L \leq 0})/f_2(D)$ is close to 0.

\begin{definition} \label{tarmac}
If $D\xrightarrow{\pi}X$ is a map of simplicial 2-complexes, define the pure 2-complexes
$X^\pi_i\subseteq X^\pi_{i-1}\subseteq\ldots \subseteq X^\pi_0\subseteq X$ with $$F_2(X^\pi_i)=\{z \in F_2(X)\|\pi^{-1}(z)|\geq i\}.$$ 
For each $i$ enumerate the connected components of
$X^\pi_i$ by $\{X^\pi_{i,j}\}_j$. Let $Q^\pi$ be the union of the index sets
of the connected components of the $X^\pi_{i}$. Define
$$Q^\pi_{L\leq 0}=\{(i,j)\in Q^\pi:\ L(X^\pi_{i,j})\leq 0 \}$$
and
$$\bar Q^\pi_{L> 0}=\{(i,j)\in Q^\pi:\ \forall (i',j') \text{ with $X^\pi_{i,j} \subset X^\pi_{i',j'}$ we have $L(X^\pi_{i',j'})> 0$} \}$$
Then define
$$X^\pi_{L\leq 0}=\bigcup_{(i,j)\in Q^\pi_{L \leq 0}}X^\pi_{i,j}\subseteq X$$
and a pure subcomplex
 $D_{L^\pi\leq 0}\subseteq D$
with 
 $$F_2(D^\pi_{L\leq 0})=\{d \in F_2(D) \mid \pi(d) \in X^\pi_{L\leq 0}\}.$$
\end{definition}

\begin{lemma} \label{boundary}
Any minimal filling $(C_r\xrightarrow{b}D\xrightarrow{\pi}X)$ satisfies
$$r\geq \sum_{(i,j) \in \bar Q^\pi_{L>0}} L(X^\pi_{i,j}).$$
\end{lemma}

Note that the lemma also holds with $\bar Q^\pi_{L>0}$ replaced by any order ideal in $Q^\pi$.

\begin{proof}
For every edge $e\in F_1(X)$ and $Q' \subset Q^\pi$ define
$$ |e|_{\infty}^{Q'}
    =\underset{f \in F_2(X):e \in \partial f}{\hbox{max}}{}\bigg| \big\{ (i',j') \in Q':\ f \in F_2(X^\pi_{i',j'})\big\}\bigg| $$
and
$$ |e|_{1}^{Q'}
    =\underset{f \in F_2(X):e \in \partial f}
    {\sum}\bigg|\big\{ (i',j') \in Q':\ f \in F_2(X^\pi_{i',j'})\big\}\bigg|. $$

By Lemma \ref{embedding} the filling $(C\xrightarrow{b}D\xrightarrow{\pi}X)$ is a 1-immersion.  Thus for any $e\in F_1(X)$
$$ f_1(\pi^{-1}e \cap \partial D)\geq\max\{0,\left(2|e|_{\infty}^{Q^\pi} - |e|_{1}^{Q^\pi}\right)\}.$$

If $e\in\partial g\cap\partial h$ with $\{g,h\}\subseteq F_2(X^\pi_{i,j})$, $g\not=h$ and $(i,j)\not\in \bar Q^\pi_{L>0}$ then 
the maximum in the definition of $|e|^{\bar Q^\pi_{L>0}}_\infty$ is achieved by both $g$ and $h$ and $2|e|_\infty^{\bar Q^\pi_{L>0}}-|e|^{\bar Q^\pi_{L>0}}_1 \leq 0$.  

It follows that for any $e\in F_1(X)$ there is
$$\max\{0,\left(2|e|_{\infty}^{Q^\pi} - |e|_{1}^{Q^\pi}\right)\}\geq \max\{0,\left(2|e|_{\infty}^{\bar Q^\pi_{L>0}} - |e|_{1}^{\bar Q^\pi_{L>0}}\right)\}.$$

Putting this together we get
\begin{eqnarray*}
 r &=& \sum_{e\in F_1(X)} f_1(\pi^{-1}e \cap \partial D)\\
          &\geq& \sum_{e\in F_1(X)} \max\{0,\left(2|e|_{\infty}^{Q^\pi} - |e|_{1}^{Q^\pi}\right)\}\\
          &\geq& \sum_{e\in F_1(X)} \max\{0,\left(2|e|_{\infty}^{\bar Q^\pi_{L>0}} - |e|_{1}^{\bar Q^\pi_{L>0}}\right)\}\\
          &\geq& \sum_{e\in F_1(X)} \left(2|e|_{\infty}^{\bar Q^\pi_{L>0}} - |e|_{1}^{\bar Q^\pi_{L>0}}\right)\\
          &=& \sum_{e\in F_1(X)}\left(\sum_{(i,j)\in \bar Q^\pi_{L>0}:\ e \in F_1(X^\pi_{i,j})} \left(2 - f^2_e(X^\pi_{i,j})\right)\right)\\
          &=& \sum_{(i,j)\in \bar Q^\pi_{L>0}}\left(\sum_{e\in F_1(X^\pi_{i,j})} \left(2 - f^2_e(X^\pi_{i,j})\right)\right)\\
          &=& \sum_{(i,j)\in \bar Q^\pi_{L>0}}L(X^\pi_{i,j}) .
\end{eqnarray*}

\end{proof}

\begin{definition} \label{jfk}
If $(C_r\xrightarrow{b}D\xrightarrow{\pi}X)$ is a minimal filling 
define (using the notation above and from Definition \ref{cambria})
$$X^{\pi\infty}_{L\leq 0}=K_{X^\pi_{L\leq 0}}(X)\subseteq X.$$

Similarly take $D^{\pi\infty}_{L\leq 0}$ to be the maximal pure two dimensional subcomplex of $\pi^{-1}(X^{\pi\infty}_{L\leq 0})\subset D.$
\end{definition}
 
\begin{lemma} \label{victrola} For any minimal filling $(C\xrightarrow{b}D\xrightarrow{\pi}X)$

$$L(X^\pi_{L\leq 0})\leq 0\text{ and }L(X_{L\leq 0}^{\pi\infty})\leq 0.$$
\end{lemma}

\begin{proof}
The $X^\pi_{i,j}$ have a natural tree structure generated by containment (that is for every $(i,j),(i',j')\in Q^\pi$ there is $X^\pi_{i,j}\cap X^\pi_{i',j'}\in\{\emptyset,X^\pi_{i,j},X^\pi_{i',j'}\}$). 
Since $L$ is additive on disjoint unions we get $L(X^\pi_{L\leq 0}) \leq 0.$

By the definition of $X^{\pi\infty}_{L\leq 0}$ every edge
 $$e \in F_1(X_{L\leq 0}^{\pi\infty})\setminus F_1(X^\pi_{L\leq 0})$$
has $f^2_e(X_{L\leq 0}^{\pi\infty})\geq 2$.

Thus
\begin{eqnarray*}
 L(X_{L\leq 0}^{\pi\infty})
 &=&   \sum_{e\in F_1(X_{L\leq 0}^{\pi\infty})}(2-f^2_e(X_{L\leq 0}^{\pi\infty}))\\
 &\leq& \sum_{e\in F_1(X^\pi_{L \leq 0})}(2-f^2_e(X_{L\leq 0}^{\pi\infty}))\\
 &&+\sum_{e\in F_1(X_{L\leq 0}^{\pi\infty}) \setminus F_1 (X^\pi_{L\leq 0})}(2-f^2_e(X_{L\leq 0}^{\pi\infty}))\\
&\leq& L(X^\pi_{L \leq 0})\\
  &\leq& 0.
  \end{eqnarray*}
\end{proof}

\begin{lemma} \label{berkeleyone}
 For every $\epsilon>{1\over 2}$ there is some $\beta$ so that every minimal filling 
$(C_r\xrightarrow{b}D\xrightarrow{\pi}X)$ with
\begin{itemize}
   \item $e(X)\geq\epsilon$ and
   \item $\chi(Z) \leq 1$  for every connected $Z\subset X$,
\end{itemize}
satisfies
 \be \label{pisa}
 f_2(D) < \beta (r + f_2(D \setminus D_{L \leq 0})).
 \ee
 \end{lemma}

\begin{proof}

By Definition \ref{jfk}
we have the complexes $X^\pi_{L\leq 0} \subset X^{\pi\infty}_{L\leq 0}\subset X$
and $D_{L\leq 0}^{\pi\infty} \subset D$.

First show that $(D-D_{L\leq 0}^\infty)\cup\hbox{Im}(b)$ is connected.  

If not, take $C$ to be a connected component different from the one containing $\hbox{Im}(b)$ 
and check that the pure two dimensional complex $\overline{\pi(C)}\cup X^{\pi\infty}_{L\leq 0}$ 
satisfies the conditions in Definition \ref{cambria}, 
so that $\pi(C)\subseteq K_{X^{\pi\infty}_{L\leq 0}}X^{\pi\infty}_{L\leq 0}=X^{\pi\infty}_{L\leq 0}$, which will be a contradiction.  
The condition is checked by taking $e$ any edge in $C$ with $\pi(e)$ in $\pi(C)-X^{\pi\infty}_{L\leq 0}$ 
and $f_{\pi(e)}^2(\pi(C)\cup X^{\pi\infty}_{L\leq 0})=1$ and constructing a smaller filling.  
Since $C$ is a manifold, $f_e^2(C)=2$ and these two triangles have the same image under $\pi$ 
so they can be contracted to a length two path, contradicting the minimality of the filling.  

This implies that each of the connected components $D_j$ of $D_{L\leq 0}^{\pi\infty}$ is a closed disk.  
Denote by $(C_{r_j}\xrightarrow{b_j}D_j\xrightarrow{\pi|_{D_j}}X)$ a filling with $L(\partial D_j)=r_j$.  
This factors through a filling $(C_{r_j}\xrightarrow{b_j}D_j\xrightarrow{\pi_j}X^{\pi\infty}_{L\leq 0})$ 
and write $\gamma_j=\pi_j b_j$.  

By Lemma \ref{victrola} we have that $L(X^{\pi\infty}_{L\leq 0})\leq 0$ and
by assumption $\chi(X^{\pi\infty}_{L\leq 0}) \leq 1$ so that by Lemma \ref{below} 
there is $\beta\geq 1$ depending only on $\epsilon$ with
$$\beta L(\gamma_j) > A(\gamma_j). $$

By the definition of the $D_j$, $b_j$ is injective and every edge of $\hbox{Im}(b_j)$ is in either $\hbox{Im}(b)$ 
or the boundary of a triangle in $D\setminus D_{L\leq 0}^{\pi\infty}.$
Thus
$$\sum_{j}r_j
  \leq r+3f_2(D \setminus D^{\pi\infty}_{L \leq 0})$$
  or
\be \label{train}
r \geq  \sum r_j -3f_2(D\setminus D^{\pi\infty}_{L\leq 0}).
\ee
By the definition of $X_{L\leq 0}^{\pi\infty}$ we have $X^\pi_{L\leq 0} \subset X_{L\leq 0}^{\pi\infty}$ so
$$f_2(D\setminus D^{\pi\infty}_{L\leq 0})
 \leq f_2(D \setminus D_{L \leq 0}). $$
Thus multiplying (\ref{train}) by $\beta$ we get
 \begin{eqnarray*}
 \beta r &\geq& \sum \beta r_j-3\beta f_2(D \setminus D^\infty_{L \leq 0})\\
 &>&
  \sum A(\gamma_j)-3\beta f_2(D \setminus D^\infty_{L \leq 0}) \\
 &=& f_2(D)-f_2(D \setminus D^\infty_{L \leq 0})-3\beta f_2(D \setminus D^\infty_{L \leq 0}) \\
 &\geq& f_2(D) -4\beta f_2(D \setminus D^\infty_{L \leq 0})\\
 &\geq& f_2(D) -4\beta f_2(D \setminus D_{L \leq 0})
\end{eqnarray*}
which proves the lemma.
\end{proof}

\begin{lemma} \label{berkeleytwo}
If $\epsilon>{1\over 2}$ and $(C_r\xrightarrow{b}D\xrightarrow{\pi}X)$ is a minimal filling with 
$e(X)\geq \epsilon$ then 
\begin{equation}\label{carsale}
f_2(D\setminus D_{L\leq 0}) \leq \frac{3}{2\epsilon-1}r .
\end{equation}
\end{lemma}

\begin{proof}
We use Lemma \ref{popped} with $w=0$ to get
that for each $(i,j) \in \bar Q^\pi_{L>0}$
 \be L(X^\pi_{i,j}) + 2\chi(X^\pi_{i,j}) \geq f_2(X^\pi_{i,j})(2e(X^\pi_{i,j})-1). \label{siege} \ee
 By assumption we have that
$\chi(X^\pi_{i,j})\leq 1$, $L(X^\pi_{i,j}) \geq 1$ and $e(X^\pi_{i,j})\geq\epsilon$, so
\begin{align*}
3L(X^\pi_{i,j}) & \geq L(X^\pi_{i,j})+2\\
&\geq L(X^\pi_{i,j})+2\chi(X^\pi_{i,j})\\ 
&\geq(2\epsilon-1) f_2(X^\pi_{i,j}).
\end{align*}
Thus by Lemma \ref{boundary},
\begin{align*}
 3r & \geq \sum_{Q^\pi_{L> 0 }} 3L(X^\pi_{i,j})\\
 & \geq \sum_{Q^\pi_{L> 0 }} (2\epsilon-1) f_2(X^\pi_{i,j})\\
 & = (2\epsilon-1) f_2(D\setminus D_{L\leq 0}).
\end{align*}
\end{proof}

\begin{pfoflem}{\ref{huji}} \label{proof_huji} \label{malia}
By Lemma \ref{topology} we can find a subcomplex $Z \subset X$ such that 
every connected $Z' \subset Z$ has $\chi(Z') \leq 1$, 
$\gamma_Z:C\rightarrow Z$ composed with the inclusion is $\gamma$ and $\gamma_Z$ 
is contractible (in $Z$).  Let $(C\xrightarrow{b}D\xrightarrow{\pi}Z)$
be a minimal filling of $\gamma_Z$.  By the definition of $e$ we have
$$e(Z) \geq e(X) \geq \epsilon$$
so the hypotheses of Lemmas \ref{berkeleyone}  and \ref{berkeleytwo} apply to $\gamma_Z$.

The area of $\gamma$ in $X$ is at most the area of
$\gamma$ in $Z$ so $$A(\gamma) \leq A(\gamma_Z)=f_2(D).$$
Combining Lemmas \ref{berkeleyone} and  \ref{berkeleytwo} we get
\begin{align*}
A(\gamma) & \leq  f_2(D)\\
& \leq \beta r + \beta f_2(D \setminus D_{L \leq 0})\\
& \leq \beta r + \beta\frac{3}{2\epsilon-1}r \\
& <\frac{5\beta}{ 2\epsilon-1}r.
\end{align*}
\end{pfoflem}

\begin{pfoflem}{\ref{hiji2}}  \label{proof_hiji2}
If $\id$ is contractible in $X$ then by Lemma \ref{topology}
there is $Z \subset X$ such that every connected $Z' \subset Z$ has 
$\chi(Z') \leq 1$ and $\id$ is
contractible in $Z$.    Let $(C\xrightarrow{b}D\xrightarrow{\pi}Z)$ be a minimal filling
of $\id$ in $Z$. By Lemma \ref{embedding}, $\pi$ is a $1$-immersion so that no images of interior edges contribute positively to $L$ and
$$L(\hbox{Im}(\pi)) \leq L(D) \leq 3.$$
By Lemma \ref{popped}  we have that
\begin{eqnarray*}
f_2(\hbox{Im}(\pi))
 &\leq& \frac{2\chi(\hbox{Im}(\pi))-2\cdot3 +L(\hbox{Im}(\pi))}{2e_3(\hbox{Im}(\pi))-1}\\
 &\leq& \frac{2\cdot 1-2\cdot3 +3}{2e_3(X)-1}\\
 &<& 0.
\end{eqnarray*}
This is a contradiction and $\id$ is not contractible in $X$.
\end{pfoflem}

{\bf \large Acknowledgements\\ }

The authors thank Nati Linial, Roy Meshulam, and Omer Angel for
helpful conversations. We also acknowledge Rick Kenyon for
mentioning the work of Ollivier and the connection to Gromov's theory of random
groups.  We would like to especially thank an anonymous referee, Gundart Anna and Uli Wagner for carefully
reading earlier drafts of this paper and making many helpful remarks.

This work supported in part by NSA grant \#
H98230-05-1-0053, NSF grant \# DMS-0501102, an AMS Centennial Fellowship, and the University of
Washington's NSF-VIGRE grant \# DMS-0354131. We would also like to thank MSRI and
the Institute for Advanced Studies at the Hebrew University of
Jerusalem where some of the research was done.

\clearpage

\section*{Appendix 1:  Connectivity of $G(n,p)$ away from the threshold} \label{appendix1}
\begin{pfoflem}{\ref{TY}}
We show  that if $ p = \frac{3 \log{n} + c }{n}$ then the probability that the
graphs $\mbox{lk}_{Y}(a) \cap \mbox{lk}_{Y}(b)$ are connected for all pairs $\{a,b \}$ is bounded below by $1-Ce^{-c}$ with
$C$ independent of $c$.  Since the probability that $TY(a,b)$ is
connected for all $a$ and $b$ in $[n]$ is increasing in $p$ this is
enough to prove the condition occurs a.a.s.\  These methods are typical in random graph theory; see for example \cite{Bollo} Theorem 7.3.

If $G$ is a graph with $n-2$ vertices with no connected components
with $k$ vertices then $G$ is connected.

For $k$ between $1$ and $n$ let $E_k$ be the expected number of
connected components in $TY(a,b)$ with $k$ vertices. We will show
that
$$\sum_{a,b}\sum_{k=1}^{\lfloor n/2 \rfloor}E_k
    = {n \choose 2} \sum_{k=1}^{\lfloor n/2 \rfloor} E_k \leq Ce^{-c}$$
for some $C<\infty.$  Thus by the union bound and the remark above,
all $TY(a,b)$ are connected a.a.s.

For any set of vertices $\{ a, b, u, v \}$, $u$ is adjacent to $v$ in $TY(a,b)$
if and only if $\{ a,u,v \}$ and $\{ b,u,v \}$ are both faces of $Y$, which
happens with probability $p^2$ by independence. So the probability
that $x$ is an isolated vertex in $TY(a,b)$ is $(1-p^2)^{n-3}$, and
we have that
\begin{eqnarray*}
E_1 & = & (n-2)(1-p^2)^{n-3}\\
& = & (n-2)\left(1-\frac{3\log{n}+c}{n}\right)^{n-3} \\
&<& C(n-2)e^{-3\log{n}+c}\\
&<& Ce^{-c}/n^2\\
&<& Ce^{-c}
\end{eqnarray*}
for some constant $C<\infty$.

The expected number of connected components in $TY(a,b)$ of order
$2$, is
\begin{eqnarray}
E_2
 &<&{n-2 \choose 2} p^2 (1-p^2)^{2(n-4)}\\
 &<& n^2 \left(\frac{3\log{n}+c}{n}\right)^2\left(1-\frac{3\log{n}+c}{n}\right)^{2(n-4)}\nonumber \\
 &< & C n^2    e^{-2(3\log{n}+c)} \nonumber \\
    &< & Ce^{-2c}/n^4 \nonumber\\
&<& Ce^{-c}.\label{hines} \nonumber
\end{eqnarray}

Similarly, since the number of spanning trees on a fixed set of
$k$ vertices is $k^{k-2}$,

\begin{eqnarray*}
\sum_{a,b}\sum_{k=3}^{\lfloor n/2 \rfloor}E_k
 &\leq &{n \choose 2} \sum_{k=3}^{\lfloor n/2 \rfloor}E_k\\
 &\leq &{n \choose 2} \sum_{k=3}^{\lfloor n/2 \rfloor} 
        {n-2 \choose k} k^{k-2} p^{2(k-1)}(1-p^2)^{k(n-k-2)} \\
 &\leq & \left(\frac{n^2}{2}\right) \sum_{k=3}^{\lfloor n/2 \rfloor} \frac{n^k}{k!}
         k^{k-2} p^{2(k-1)} e^{-p^2k(n-k-2)} \\
 &\leq & \left(\frac{n^2}{2}\right) \sum_{k=3}^{\lfloor n/2 \rfloor}
        k^{-5/2} e^k
        n^k p^{2(k-1)} e^{-p^2k(n-k-2)} \\
 &\leq &  \left(\frac{n^3}{2}\right) \sum_{k=3}^{\lfloor n/2 \rfloor}
        k^{-5/2} \exp \bigg[k+(k-1)\log{3} +
        (k-1)\log{\log{n}}\\
 &&     \hspace{2.2in}   -3 k(n-k-2)\log{n}/ n\bigg] \\
 &\leq &  \left(\frac{n^3}{2} \right)\sum_{k=3}^{\lfloor n/2 \rfloor} k^{-5/2}
        \exp [-7k\log{n}/5] \\
 &\le &  n^{-6/5}\\
&<& Ce^{-c}.
\end{eqnarray*}

For the second condition note that for fixed $a,b,d \in [n]$ we have that
$$\prob(abd \not \in F_2(Y))=1-p.$$
For each $d$ this is independent.  So for a fixed $a,b \in [n]$
$$\prob(\not \exists d:\ abd \in F_2(Y))=(1-p)^{n-2}=O(1-\frac{1}{n^{1/2}})=O(e^{-n^{1/2}}).$$
Then the union bound shows that the second condition is satisfied a.a.s.
\end{pfoflem}

\newpage

\section*{Appendix 2:  Local-to-global} \label{appendix2}

In this appendix we prove a local-to-global theorem for linear isoperimetric inequalities.  The statement and proof are similar in spirit to results already appearing for groups \cite{gromov1,Pap}, but we need the result for simplicial complexes so we include a proof here for the sake of completeness.

Throughout the section we fix $\epsilon \in (0,.25)$ and 
work with simplicial complexes scaled so that edges have length $\epsilon$ and triangles have area $\epsilon^2$. 
\begin{theorem}\label{pap}
If $X$ is a simplicial complex with edge lengths $\epsilon$ and triangle areas
$\epsilon^2$ and there is an $n\geq 1$ such that every loop $\gamma$ with $1\leq A\gamma\leq 44$
has $A\gamma<({L\gamma\over 44})^n$ then every contractible loop $\gamma$
with $1\leq A\gamma$ has $A\gamma< L\gamma$.
\end{theorem}
Theorem \ref{gromov} follows easily from this result.
The key concept that we will use in the proof of Theorem \ref{pap} is that of a {\bf shortcut}.

\begin{definition}
Let  $([k]\xrightarrow{x}C_r\xrightarrow{b}D\xrightarrow{\pi}X)$ is a $k$-marked 
filling if $(C_r\xrightarrow{b}D\xrightarrow{\pi}X)$ is a filling and 
$[k]\xrightarrow{x}C$ is a cyclically order preserving map from $[k]$
to $F_0(C_r)=[r]$ and $x(i)$ is denoted $x_i$. (This simply means that there is $a\in[k]$ so that $x_{i}<x_{i+1}$ if $i\not= a$ and 
$x_k<x_1$ if $k\not= a$.)
For each marked filling we define the covering by cyclically order preserving paths 
$J^x_i:I_{x_{i+1}-x_i}\rightarrow C_r$ from $x_i$ to $x_{i+1}$ and 
$J^x_k:I_{x_1-x_k+r}\rightarrow C_r$ from $x_k$ to $x_1$.  Define paths 
$b^x_i=b(J^x_i):I_{x_{i+1}-x_i}\rightarrow D$ and 
$\gamma^x_i=\gamma^R_i=\pi(b^x_i):I_{x_{i+1}-x_i} \rightarrow X$.
\end{definition}

\begin{definition}
A {\bf shortcut} is a $2$-marked filling
$$([2]\xrightarrow{y}C\xrightarrow{b}D\xrightarrow{\pi}X)$$ with
$$d\big(y_1,y_2\big)-d\big(by_1,by_2\big)\geq 1.$$
Fix a (geodesic) path $B^y:I_{d(by_1,by_2)}\rightarrow D$ from $by_1$ to $by_2$ 
and cycles $B^y_i=b^y_i\cdot \overline{B^y}:C\rightarrow D$ and 
$\Gamma^y_i=\pi B^y_i:C\rightarrow X$.  

We say that a shortcut is of {\bf  type $\mu$} if
$$ d\big(y_1,y_2\big)\geq \mu \ \ \ {\text and }\ \ \ d\big(by_1,by_2\big) \leq \mu-1.$$
\end{definition}

Note that every shortcut has at least one type.  

We prove Theorem \ref{pap} by induction on the minimal area of a filling.

Throughout the rest of this section we let $\alpha=44$  and $\mu={13\over 2}$.  The first step is to show that all shortcuts in a minimal counterexample are long.

\begin{lemma}
If $$([2]\xrightarrow{y}C\xrightarrow{b}D\xrightarrow{\pi}X)$$ be a shortcut in a filling which is a minimal area counterexample to Theorem \ref{pap} (that is $A(D)$ is minimal among all fillings with $1\leq A(D)=A(\pi b)\geq L(\pi b)$) then $A(\Gamma^y_1)\geq 1$ and $L(\Gamma^y_1)\geq\alpha$.
\end{lemma}
\begin{proof}
Write $\gamma=\pi b$.  
First note that $\alpha< A(\gamma)$.  Otherwise, as the shortcut is a counterexample to Theorem \ref{pap} and $A(\gamma) < \alpha$ there is 
$$L\gamma\leq A\gamma<\left({L\gamma\over\alpha}\right)^n$$

So $$\alpha^{n}< (L\gamma)^{n-1}\leq (A\gamma)^{n-1}\leq \alpha^{n-1}$$ which is a contradiction.

Note that $1\leq A\gamma_{1}$.  Otherwise, 
$$A\gamma_{2}=A\gamma-A\gamma_{1}>\alpha-1\geq 1.$$ Since $\gamma$ is a minimal area counterexample we must have that $\gamma_{i}$ is not a counterexample so $A\gamma _{i}<L\gamma_{i}$.
Hence by the definitions of the $\gamma_{i}$ and of a shortcut
$$1>A\gamma_{1}=A\gamma-A\gamma_{2}>L\gamma-L\gamma_{2}=LJ_1-LB=d(x_1,x_2)-d(bx_1,bx_2)\geq 1,$$ a contradiction.

Finally, if $L\gamma_{1}\leq \alpha$ then 
$$1\leq A\gamma_{1} < L\gamma_{1}\leq \alpha$$ so that by the hypothesis of Theorem \ref{pap} there is 
$1\leq A\gamma_{1}<({L\gamma_{1}\over\alpha})^n$ and $\alpha < L\gamma_{1}$, a contradiction.
\end{proof}

\begin{definition}
A {\bf rectangle} of type $(u_1,u_2)$ is a $4$-marked filling $([4]\xrightarrow{x}C\xrightarrow{b}D\xrightarrow{\pi}X)$ with
$$d_D(\hbox{Im}(b^x_i),\hbox{Im}(b^x_{i+2}))\geq u(i)$$ for both $i\in\{1,2\}$.
\end{definition}


\begin{definition}
The {\bf ball} $B(r,A)$ of radius $r$ about $A\subseteq X$ (a metric space) is
all points with distance to $A$ at most $r$.
\end{definition}

\begin{definition}
If $R=([4]\xrightarrow{x}C\xrightarrow{b}D\xrightarrow{\pi}X)$ is a rectangle with type $u_j>r$ then
the {\bf $r$-neighborhood} of the $i=j$ (or $i=j+2$) edge of $R$ is the subrectangle
$$N_{r,i}(R)=\left([4] \xrightarrow{x'}C'\xrightarrow{b'}D'\xrightarrow{\pi|_{D'}}X\right).$$

Take $y_1, y_2\in b^{-1}B(r,\hbox{Im}(b^x_i))$ maximizing the length $t$ of the minimal path
$a:I_t\rightarrow C$ from $y_1$ to $y_2$ containing $\hbox{Im}(J^x_i)$ but not intersecting 
$\hbox{Im}(J^x_{i\pm 2})$.  
Take $x'_i=x_i$, $x'_{i+1}=x_{i+1}$, $x'_{i\pm 2}=y_1$
and $x'_{i+1\pm 2}=y_2$.
Take $z:I_s\rightarrow D$ to be the $1$-immersed path from $by_2$ to $by_1$ with interior avoiding 
$\hbox{Im}(b)$ along the boundary (interior to $D$) of $B(r,\hbox{Im}(b^x_i))$.  
Take $b'=a\cdot \overline{z}:C'=C_{t+s}\rightarrow D'\subseteq D$.  
\end{definition}

\begin{lemma} \label{disjoint}
If $R$ is a rectangle of type $(r+s,v)$ then $N_{r,1}(R)$ is of type $(r,v)$ and has
interior disjoint from $N_{s,3}(R)$.
\end{lemma}
\begin{proof}

If there is a point in $N_{r,1}(R)^\circ \cap N_{s,3}(R)^\circ$ then there is a path from $\hbox{Im}(b_{1})$ to $\hbox{Im}(b_{3})$ of length 
less than $r+s$.  
\end{proof}

\begin{lemma} \label{rectanglearea}
If $R$ is a rectangle of type $u$ then $A(R)\geq 2u(1)u(2)$.
\end{lemma}
\begin{proof}
Since any type $(\epsilon,\epsilon)$ rectangle has at least $2$ triangles,
an easy induction using \ref{disjoint} now
shows that any type $(u(1),u(2))$ rectangle has area at least $2u(1)u(2)$.
\end{proof}

\begin{lemma} \label{maicon}
Let $$F=(C\xrightarrow{b}D\xrightarrow{\pi}X)$$ be a minimal filling which is a minimal area 
counterexample to Theorem \ref{pap}.  Then there exists a filling
$$F'=(C'\xrightarrow{b'}D'\xrightarrow{\pi |_{D'}}X)$$ with 
$A\gamma \geq A\gamma' >\alpha$ and $A\gamma'>1.15L\gamma'$.
\end{lemma}

\begin{proof}
In the proof we follow the following steps.

\noindent {\bf Define a new marked filling.}

Write $\mu\leq m\epsilon<\mu+\epsilon$.

We do this in two cases.  First if $F$ has  no type $\mu$ shortcut then our filling is $F'=F$ (and for notational reasons take any 1-marking $[1]\xrightarrow{y}C=C'$).
Define a $2t$-marking $[2t]\xrightarrow{r}C$ with $r_i=1+m(i-1)$ and $t$ maximal with $LC\geq 2tm\epsilon$. 

If $F$ does have a type $\mu$ shortcut then choose one $[2]\xrightarrow{y} C$ such that $LJ^y_1$
is minimal and among those one with $AB^y_1$ minimal. Our filling $F'$ will have $b'=B^y_1$.

Now we define a $(2t)$-marking $[2t]\xrightarrow{r}C'$ of $F'$ by $r_i=x_1+mi\in\hbox{Im}(J^y_1)$, taking $t$ maximal so that $m\epsilon(2t+1)\leq LJ^y_1$. 

Note that all of the paths $b^r_i$ for $i=1,\dots,2t-1$ for this marked 
filling have image in $b^y_1$.
Note that in both cases 
$$t>\frac{Lb'-(\mu-1)-2(\mu+\epsilon)}{2(\mu+\epsilon)}.$$
{\bf Define the rectangles $R_i$.}

For any $i<j\in[t]$ consider the type $(u_1=\mu-1, u_2=\mu-1)$ rectangle $R_{i,j}$ marking $F'$ by $x_1=r_i$, $x_2=r_{i+1}$, $x_3=r_j$ and $x_4=r_{j+1}$.  The type follows from the minimality of the shortcut and Lemma \ref{disjoint}. Define $R_i=N_{\frac{\mu -1}{2},1}(R_{2i,2j})=N_{\frac{\mu -1}{2},3}(R_{2k,2i})=N_{\frac{\mu -1}{2},2}(R_{i-1,i+1})$.  By Lemma \ref{disjoint} $R_i$ has type $u_1=\frac{\mu -1}{2}$, $u_2=\mu$ and has interior disjoint from $R_j$ if $i\not= j$.  
 
By Lemma \ref{rectanglearea} and the last step we have that
$$A(R_i)\geq 2\left(\frac{\mu-1}{2}\right)(\mu-1)\geq\left(\mu-1\right)^2.$$
{\bf Finally we compute the area of $b$.}

Recalling that $\mu=6.5$, $Lb'>44$ and $\epsilon<.25$ we have that $(\mu-1)^2=30.25$ and
$$A b' \geq  \sum_{i=1}^t A(R_i)\geq t(\mu-1)^2
 \geq \left(\frac{Lb'-3\mu-2\epsilon+1}{2(\mu+\epsilon)}\right)(\mu-1)^2 \geq  \frac{30.25Lb'-575}{14} >  2.15Lb'-44>1.15Lb'.$$

\end{proof}

\begin{pfofthm}{\ref{pap}}
Assume not. Fix a counterexample $\gamma:C\rightarrow X$ with $1\leq L\gamma\leq A\gamma$ such that
$A\gamma$ is minimal as well as a minimal filling
$$F=(C\xrightarrow{b}D\xrightarrow{\pi}X)$$ of $\gamma$.

By Lemma \ref{maicon} we have that $\alpha<A\gamma>1.15 L\gamma$.  We construct a smaller counterexample.  Simply alter $D$ by removing one two face touching the boundary of $D$
and make the corresponding changes to the rest of the filling.  This increases the length of the curve by 
at most $\epsilon<.25$ and decreases its area by at most $\epsilon^2<.0625$.  It is easy to check that
this is still a counterexample.  As $\gamma$ was minimal this is a contradiction.
\end{pfofthm}

\begin{pfofthm}{\ref{gromov}}
Rescale $X$ so that edges have length ${1\over \rho}$ and triangles have area
$({1\over \rho})^2$ and apply Theorem \ref{pap} with $n=1$ and $\epsilon={1\over \rho}$.
\end{pfofthm}

Theorem \ref{pap} is closely related to the gap between quadratic and linear growth, as discussed for instance in \cite{Bowditch}.  In fact this gap follows from Theorem \ref{pap} with $n=2$.  

\bibliography{bhk}{}
\bibliographystyle{plain}

\end{document}